\documentclass[10pt]{article}
\usepackage[utf8]{inputenc}
\usepackage[margin=1.44in]{geometry}
\usepackage{bm}
\usepackage{comment}
\usepackage{lscape}
\usepackage{float,enumitem}
\usepackage{amsmath,amssymb,color,CJK,mathrsfs}
\usepackage{amsthm,multirow}
\usepackage{algorithm, algpseudocode,hyperref}
\usepackage[title]{appendix}
\usepackage{graphicx,natbib}
\usepackage{subcaption}
\usepackage{authblk}

\allowdisplaybreaks

\def\d{\mathop{\rm d\!}\nolimits}

\def\llambda{\mathop{\boldsymbol\lambda}\nolimits}
\def\M{M}
\def\x{\bm{x}}
\def\pr{\text{pr}}
\def\y{\bm{y}}
\def\z{\bm{z}}
\def\F{\mathrm{F}}

\newtheorem{theorem}{Theorem}[section]

\newtheorem{lemma}[theorem]{Lemma}

\theoremstyle{definition}
\theoremstyle{remark}

\newtheorem{remark}{Remark}[section]
\numberwithin{equation}{section}
\theoremstyle{remark}

\newtheorem{proposition}{Proposition}[section]

\numberwithin{equation}{section}

\date{}
\begin{document}

\title{Empirical likelihood for Fr\'echet means on open books}
\author{Karthik Bharath}
\author{Huiling Le} 
\author{Xi Yan}
\affil{School of Mathematical Sciences, The University of Nottingham}
\author{ Andrew T A Wood}
\affil{Australian National University}



\maketitle

\begin{abstract}

Empirical Likelihood (EL) is a type of nonparametric likelihood that is useful in many statistical inference problems, including confidence region construction and $k$-sample problems. It enjoys some remarkable theoretical properties, notably Bartlett correctability. One area where EL has potential but is under-developed is in non-Euclidean statistics where the Fr\'echet mean is the population characteristic of interest. Only recently has a general EL method been proposed for smooth manifolds. In this work, we continue progress in this direction and develop an EL method for the Fr\'echet mean on a stratified metric space that is not a manifold: the open book, obtained by gluing copies of a Euclidean space along their common boundaries. The structure of an open book captures the essential behaviour of the Fr\'echet mean around certain singular regions of more general stratified spaces for complex data objects, and relates intimately to the local geometry of non-binary trees in the well-studied phylogenetic treespace. We derive a version of Wilks' theorem for the EL statistic, and elucidate on the delicate interplay between the asymptotic distribution and topology of the neighbourhood around the population Fr\'echet mean. We then present a bootstrap calibration of the EL, which proves that under mild conditions, bootstrap calibration of EL confidence regions have coverage error of size $O(n^{-2})$ rather than $O(n^{-1})$. 
\end{abstract}

\section{Introduction}
\subsection{Background and motivation}
The Fr\'{e}chet mean of a probability distribution $\mu$ is arguably the cornerstone of data analysis on a metric space $(\M,d)$. When $M$ is a smooth finite-dimensional manifold, considerable effort has been made to understand how the geometry of $M$ influences existence and uniqueness of the Fr\'echet mean \citep{HK, BA}, and to identify corresponding conditions that ensure consistency and a Central Limit Theorem (CLT) for its empirical version based on a random sample from $\mu$ \citep{BP1}; the limiting distribution is a Gaussian with full support on the tangent space of the population Fr\'{e}chet mean.  Theoretical complications in the CLT theory can arise even in the case of a smooth compact manifold; see \cite{EGHT} and \cite{HLW}.

An important example of a metric space in modern data analysis is a stratified space. These space are disjoint unions of manifolds, known as the strata, glued isometrically along their common boundaries (singularities), and occur naturally as representation spaces for various complex data objects, including trees, graphs, persistence diagrams from Topological Data Analysis, and quotients of manifolds under isometric group actions \citep{feragen2020statistics}. The study of Fr\'echet means on a stratified space $M$, and their corresponding asymptotic sampling properties, has been carried out by \cite{HHLMMMNOPS} when $M$ is the open book; by \cite{barden2013central} when $M$ is the phylogenetic treespace; and, more recently by \cite{mattingly2023central} for more general stratified spaces. In contrast to the manifold setting, the topology of the stratified space, owing to its singularities, plays a fundamental role in determining the asymptotic behaviour of the Fr\'{e}chet mean: it could exhibit a tendency to `stick' to lower-dimensional subspaces, resulting in Gaussian limiting distributions with support on a lower-dimensional subspace; or it could oscillate between full- and lower-dimensional subspaces, leading to a limit mixture distribution. 

The current state-of-the-art for asymptotic inference for the Fr\'{e}chet mean on manifolds and stratified spaces rests on using the above limiting Gaussian distributions. Computing, or approximating, the asymptotic covariance of the Gaussian is difficult in practice, and this difficulty has stymied progress towards inference.

In the Euclidean setting, inference using the empirical likelihood (EL) offers an efficient nonparametric alternative to traditional parametric likelihood-based inference \citep{ABO}. Central to the EL method is the Wilks' theorem that establishes asymptotic chi-square distribution for the EL ratio function, free from the asymptotic covariance parameter. Moreover, it is well-known that confidence regions for a population parameter can be constructed using the EL with asymptotic coverage accuracy as good as that for parametric likelihoods \citep[e.g,][]{ABO88}.

In directional and shape statistics, where $M$ is a sphere or its quotient (e.g. real or complex projective spaces), the Euclidean EL method for testing hypotheses on, and constructing confidence regions for, the extrinsic Fr\'{e}chet mean was successfully adapted by viewing $M$ as an embedded submanifold of a higher-dimensional ambient Euclidean space \citep{FHJW96, AMWO}. More recently, an EL method for the intrinsic Fr\'{e}chet mean on general manifolds $M$ was developed \citep{KO}. 

Sticky behaviour of the Fr\'echet mean on a stratified space $M$ is a major obstacle in developing an EL method for inference: the local structure of $M$ in and around its singularities can be quite varied and subtle, thus making it to difficult to precisely describe how the Fr\'echet mean behaves in such regions with a level of generality required for the development of an asymptotic theory; as such, the path towards a general EL framework on stratified $M$ is unclear. 

The purpose of this article is to make progress in this direction by focussing on a simple stratified $M$, the open book: a space obtained by gluing disjoint copies of a half spaces in $\mathbb R^p$ along their boundary hyperplanes, such that the singularities are of codimension one. We focus on the open book for the following reasons. First, every stratified space that is singular along a stratum of codimension one is locally homeomorphic to the open book \citep{goresky1988stratified}; developing an EL method for the open book that accommodates a sticky Fr\'echet mean on its codimension one strata will shed light on the challenges, and corresponding mitigation strategies, when moving onto general stratified spaces with strata of codimension greater than one \citep{BL}. Second, from a methodological perspective, the open book relates intimately to the neighbourhood structure of 
certain non-binary trees in the space of phylogenetic trees \citep{BHV}, whose geometry has now been extensively studied, both from statistical \citep[e.g.,][]{TN, willis2019confidence} and computational perspectives \citep[e.g.,][]{miller2015polyhedral, owen2011computing}, and used in various applications involving tree-structured data \citep[e.g.,][]{feragen2013tree}. 

\subsection{Contributions}
Our strategy for developing an EL framework for the Fr\'echet mean of $\mu$ on the open book $M$ is to consider its equivalent formulation, based on an estimating equation, for the Euclidean mean at the tangent cone of the Fr\'echet mean in $M$. Definition of the EL function at a point $x \in M$ depends on the location of $x$  amongst the top-dimensional and codimension one strata. Consequently, we derive a new Wilks' theorem in Theorem \ref{thm3} and establish a limiting chi-square distribution for the EL test statistic, with support and degrees of freedom determined by the dimension of the stratum in which the population Fr\'echet mean is located, potentially exhibiting sticky behaviour. 

Proof of the Wilks' theorem requires the EL test for the multivariate mean in $\mathbb R^p$ to be consistent against all simple alternatives. We were unable to find such a result in the literature. Theorem \ref{thm:alternative} establishes this fact, and may be of independent interest. 

We then consider bootstrap calibration of the confidence regions for the Fr\'echet mean obtained from  Wilks' theorem, and in Theorem \ref{thm:bootstrap} prove that such confidence regions enjoy benefits of reduced coverage error of order $O(n^{-2})$, similar to the Euclidean case. To the best of our knowledge, such a result has not been proved even when $M$ is a manifold. We then demonstrate the utility of the asymptotic results on a simulated and real-data example involving the 3-spider, a special case of the open book, pertaining to rooted phylogenetic trees on three leaves. 

\section{Empirical likelihood on the Euclidean space}
This section serves two purposes: to review key aspects of the empirical likelihood (EL) method when $M$ is the Euclidean space $\mathbb R^p, p \geq 1$, which are relevant to the sequel; and, to derive a new result on consistency of the EL test based on Wilks' theorem against simple alternatives in $\mathbb R^p$, needed later for the open book. Exposition in this section is based on material from \cite{ABO}. 

EL is a nonparametric method of inference based on a data-driven likelihood ratio function. For a given set of i.i.d data $x_1,\ldots,x_n$ and a distribution $\mu$, both on $\mathbb R^p$, consider
\[\psi_n(\mu):=\frac{L(\mu)}{L(\mu_n)}\thickspace,\]
where $\mu_n$ denotes the empirical distribution of $x_1,\ldots,x_n$, and  $L(\mu):=\prod_{i=1}^n\mu(\{x_i\})$
is the nonparametric likelihood of $\mu$. When $\mu\not=\mu_n$, $L(\mu)<L(\mu_n)$. If the probability measure $\mu$ places probability $p_i\geqslant0$ on the value $x_i$, then $\sum_{i=1}^np_i\leqslant1$ and $L(\mu)=\prod_{i=1}^np_i$, so that
\[\psi_n(\mu)=\prod_{i=1}^n (np_i ) \thickspace .\]

\vskip 6pt
For a family $\mathcal P$ of distributions on $\mathbb R^p$ with point masses at the data, define
\[
\mathcal R_n(x):=\sup_{\mu \in \mathcal P}\left\{\psi_n( \mu)|\sum_{i=1}^nx_i\mu(\{x_i\})=x\right\}.
\]
The empirical likelihood (or profile empirical likelihood ratio function) for the mean is defined via the optimisation problem
\begin{eqnarray}
\label{eqn2}
\mathbb R^p \ni x \mapsto \log(\mathcal{R}_n(x))= \max\sum_{i=1}^n\log (n p_i),\qquad\hbox{subject to }\left\{
\begin{array}{l}
\sum\limits_{i=1}^n p_i=1;\,\,p_i\geqslant0;\\
\sum\limits_{i=1}^n p_i(x_i-x)=0.
\end{array}
\right.
\end{eqnarray} 
%
%
The asymptotic distribution of $-2\log(\mathcal R_n(x_0))$ at the true population mean $x_0$ is then obtained by the following nonparametric analogue of Wilks’ theorem. Throughout, abusing notation, we use $\chi^2_q$ to denote both a chi-square distributed random variable with $q$ degrees of freedom and the corresponding distribution; context will disambiguate the two. 

\begin{theorem}[Wilks' Theorem on $\mathbb R^p$ \citep{ABO}]
Let $x_1,\cdots,x_n$ be i.i.d. random vectors from a distribution on $\mathbb R^p$ with mean $x_0$ and finite covariance matrix of rank $q>0$. Then, as $n\rightarrow\infty$, $-2\log(\mathcal R_n(x_0))$ converges in distribution to a $\chi^2_q$ random variable.
\label{thm0}
\end{theorem}

Typically, $q=p$ and the covariance matrix is of full rank, but this is not a requirement. In practice, however, rank of the sample covariance matrix is used upon assuming that $n > p$. EL hypothesis tests for the mean reject $H_0: E[x_1]=x_0$ when $-2\log(\mathcal R_n(x_0))<r_0$, where $r_0$ is an appropriate percentile of the $\chi^2_q$ distribution. The corresponding confidence regions assume the form 
\[
C_n(r_0):=\big\{x\mid-2\log(\mathcal R_n(x))\geqslant r_0\big\}.
\] 
An attractive property of  EL in the Euclidean case  is that that confidence region $C_n(r_0)$ is a convex subset of $\mathbb R^p$ since it is a (super-)level set of the convex function $x \mapsto -2\log(\mathcal R_n(x))$. 

\vskip 6pt
Using $\bar x_n$ to denote the sample mean of $x_1,\ldots,x_n$, proof of Wilks' theorem is based on the fact that the leading term of $-2\log(\mathcal R_n(x_0))$ is 
\[
n(\bar x_n-x_0)^\top \hat V^{-1}(\bar x_n-x_0), \quad \hat{V}:=\frac{1}{n}\sum_{i=1}^n(x_i-x_0)(x_i-x_0)^\top,
\]
assuming that the population covariance matrix has full rank $p$ and $n>p$.  Note that this leading term equals the Hotelling's $T^2$ statistic with an error of order $O_p(n^{-1})$. Bootstrap calibration for the region $C_n(r_0)$ reduces its coverage error to $O(n^{-2})$ under mild moment and smoothness conditions; see Appendix B of \cite{FHJW96}. 
So far as we are aware, no results have been published  on what happens to the EL ratio $\mathcal{R}_n(x)$ in (\ref{eqn2}) when $x \neq x_0$ is a fixed alternative, where as before $x_0$ is the population mean.  We have the following result, which will play in important role when deriving a type of Wilks' theorem for the Fr\'echet mean on the open book. 

\begin{theorem}[Consistency under simple alternatives]
\label{thm:alternative}
Under the assumptions of Theorem \ref{thm0}, with $x \neq x_0$ fixed,
$\mathcal{R}_n(x) \rightarrow 0$ in probability as $n \to \infty$.
\end{theorem}


\section{The open book}
\label{sec:intrinsic}
\subsection{Background}

We use notation from \cite{HHLMMMNOPS} in what follows. Let $N=\mathbb R_{\geqslant 0} \times \mathbb R^{p-1}$ with boundary $\partial N=\{0\}\times\mathbb R^{p-1}$. The \emph{open book $\M$ with $\ell>2$ pages} is defined as the quotient of the disjoint union $N \times \{1,\ldots,\ell\}$ of closed half-spaces under the equivalence relation that identifies their boundaries.  More specifically, if $x=(0,u,j), \, y=(0,v,k) \in \partial N \times \{1, \ldots, \ell\}$, then $x \sim y$ if and only if $u=v \in \mathbb{R}^{p-1}$, regardless of whether $j=k$ or $j \neq k$. The common boundaries are said to be isometrically glued with respect to the standard Euclidean metric from the norm $\|\cdot\|$ on $\mathbb R^p$.

Each page $L_k:=N \times \{k\}, k=1\ldots,\ell$ is of dimension $p$ and pages are joined together at a $(p-1)-$dimensional spine $S$, which comprises the equivalence classes in $\cup_{k=1}^\ell (\partial N \times \{k\})$, determined using the equivalence relation defined above. The interior of a page $L_k$ is $L_k^o=L_k \backslash S$. Consequently, $\M=S \cup L_1^o\cup \cdots \cup L_\ell^o$. The space $\M$ is a stratified space comprising $\ell$ top strata, each a copy of $\mathbb R^p$. An open book with $p=1$ is referred to as the $\ell$-spider. 

The open book can be equipped with an intrinsic metric induced from the Euclidean metric on $\mathbb R^p$. For any two points $\x=(x,i)$ and $\y=(y,j)$ in $\M$, consider the metric 
\begin{equation}
\label{eq:metric}
d(\x,\y):=
\begin{cases}
\|x-y\| & \text{if } i=j;\\
\|x-Ry\| & \text{if } i \neq j,
\end{cases}
\end{equation}
where $R: \mathbb R^{p} \to \mathbb R^{p}$ is the reflection across the hypersurface $\{0\}\times\mathbb R^{p-1}$.  Specifically, if $y=(y^{(0)}, y^{(1)})$, where $y^{(0)} \geqslant 0$ and $y^{(1)} \in \mathbb{R}^{p-1}$, then $Ry = (-y^{(0)}, y^{(1)})$. We note that $d(\x,\y)$ is the intrinsic metric which represents length of the shortest path, the geodesic, in $M$ between any pair of points $x$ and $y$; evidently, geodesics between points in the interior of the same page 
are straight lines in $\mathbb R^p$. The curvature at any point, in the sense of Alexandrov, in the interior of a page is zero but on the spine is negative. Summarily, $(M,d)$ is a CAT(0) complete geodesic metric space, with global nonpositive curvature; see Appendix \ref{sec:app_cat} for relevant definitions. 

The appropriate notion of space of directions of geodesics emerging for a point $x \in M$ is that of the tangent cone, which generalizes the tangent space when $M$ is a manifold \citep[p.190][]{BH}. The tangent cone at a point $x$ in the interior of a page $L_k$  is identified with $\mathbb R^p$. However, when $x$ lies in the spine $S$, the tangent cone is $M$ itself, and not a vector space. 

Nevertheless, convenient global coordinates for a point $\x$ in $\M$ may be prescribed in the following manner: if $\x\in L_k$ then $\x=(x,k)$, where $x=(x^{(0)},x^{(1)})\in N$, $x^{(0)}\in\mathbb R_{\geqslant 0}$,  $x^{(1)}\in\mathbb R^{p-1}$ and $k\in\{1,\ldots,\ell\}$; use $\x=x^{(1)}\in\mathbb R^{p-1}$ when $\x\in S$, in order to avoid the redundant $0\in\mathbb R_{\geqslant0}$ and the arbitrariness in $k$. The spine $S$, a set of equivalence classes, is identified with $\mathbb R^{p-1}$. 
 

\subsection{Fr\'echet mean and its relationship to the Euclidean mean}
Without loss of generality, we assume that the support of $\mu$ has non-empty intersection with all pages, and $\mu$ hence decomposes as
\begin{equation}
\mu(A)=w_0 \mu_0(A \cap S)+\sum_{k=1}^\ell  w_k \mu_k(A \cap L_k^o),
\label{eq:decomposition}
\end{equation}
for any Borel set $A$ in $\M$, where $\{\mu_k, k=1,\ldots,\ell\}$ is a disintegration $\mu$ such that $\mu_k(L^o_k)=1$ for every $k$ \citep[Lemma 2.1][]{HHLMMMNOPS}.   We assume that $w_k:=\mu(L^o_k)>0$ for all $k=1,\ldots,\ell$. 

The Fr\'echet mean of $\mu$ on the open book $\M$ is defined as the minimizer of the Fr\'echet function
\begin{equation}
\label{eq:frechet}
	\M \ni x \mapsto \F_\mu(\x):=\int_{\M}d(\x,\y)^2\d\mu(\y), 
\end{equation}
where $d(\cdot, \cdot)$ is as defined in \eqref{eq:metric}. Nonpositive curvature of $M$ guarantees that the Fr\'echet mean of $\mu$, is unique whenever $\F_\mu(x) < \infty$ for some $x \in M$ \citep[Proposition 4.3][]{KTT}.

Denote by $\x_0$ the Fr\'echet mean of $\mu$. Study of $\x_0$ in $\M$ is aided by $l$ maps, one for each page in $M$, and another on the spine $S$, using which the Fr\'echet mean of $\mu$ on $M$ may be related to the Euclidean mean of the pushforward of $\mu$ under the maps. The maps plays an important role in development of the EL method on $M$. 

The \emph{folding map} $F_k:\M \to \mathbb R^p$, $k=1,\ldots,\ell$ is defined as
\[
F_k(\x)=
\begin{cases}
x & \text{if } \x=(x,k);\\
Rx& \text{if } \x=(x,j) \text{ and } j \neq k,
\end{cases}
\]
for the reflection $R$ as defined above. Then, the pushforward $\mu \circ F_k^{-1}$ for each $k$ is a distribution on $\mathbb R^p$, and $F_k(X)$ is a Euclidean random variable when $X$ is a random variable on $\M$ with distribution $\mu$. The requirement for the point $\x_0=(x_0,k)$ on page $k$ (and not on the spine) to be the Fr\'echet mean of $\mu$ is characterized by
\begin{equation}
\int_{\M}F_k(\x)\,\d\mu(\x)=F_k((x_0,k))=x_0.
\label{eqn1b}
\end{equation}
This implies that, in this case, $\x_0=(x_0,k)$ is the Fr\'echet mean of $\mu$ if and only if $x_0$ is the Euclidean mean of $\mu \circ F^{-1}_k$ on $\mathbb R^p$. The Fr\'echet mean $\x_0$ in this case is said to be \emph{non-sticky}. 

For $\x_0=(x^{(0)}_{0},x^{(1)}_{0})$ on the spine, represented by $x_{0}^{(1)}\in\mathbb R^{p-1}$, another map is required to characterize the Fr\'echet mean. Precisely, the point $\x_0=(x^{(0)}_{0},x^{(1)}_{0})$ is the Fr\'echet mean of $\mu$ if the following two conditions are satisfied:
\begin{align}
\displaystyle\int_{\M}P_s(\x)\,\d\mu(\x)&=P_s(\x_0)=x_{0}^{(1)};\nonumber \\
\displaystyle\int_{\M}\langle F_j(\x),e_j\rangle\,\d\mu(\x)&\leqslant 0,\quad j=1,\cdots,\ell,
\label{eqn1c}
\end{align}
where $P_s: M \to S$ is the projection operator such that $P_s(\x)=x^{(1)}$ denotes the projection of a point $\x = (x^{(0)}, x^{(1)}) \in \M$ on to the spine $S$, and $e_j$ is the `outward' unit vector that is tangent to page $j$ and is orthogonal to $S$ \citep{HHLMMMNOPS}. 
Under our assumption that the support of $\mu$ has non-empty intersection with all pages, the inequalities in the constraints given in \eqref{eqn1c} can include at most one equality. That is, two situations may arise: (i) for each $j=1, \ldots , \ell$, 
\[
\int_{\M} \langle F_j(\y),e_j \rangle \d\mu(\y) < 0,
\]
and the Fr\'echet mean $\x_0$ is said to be \emph{sticky}; (ii) for a single $k=1, \ldots , \ell$, 
\[
\int_{\M} \langle F_k(\y),e_k \rangle \d\mu(\y) =0,
\]
where the corresponding integral involving $F_j$ is negative for each $j \neq k$; in this case the Fr\'echet mean $\x_0$ is said to be \emph{half-sticky}.

In the former case, the $\x_0$ is the Fr\'echet mean of $\mu$ if and only if $x^{(1)}_{0}$ is the Euclidean mean of $\mu \circ P_s^{-1}$ on the spine $S\cong \mathbb R^{p-1}$. However, in the latter case, characterisation is direction-dependent, and depends on both $P_s$ and the folding the maps: with those unit vectors $u \in \mathbb{R}^p$ for which $u^\top e_k>0$, it requires the Euclidean mean of $\mu \circ F_k^{-1}$ to equal $x^{(1)}_{0}$; on the other hand, for those $u$ such that $u^\top e_k \leq 0$, it requires the Euclidean mean of $\mu \circ P_s^{-1}$ to equal $x^{(1)}_{0}$.


\section{Intrinsic empirical likelihood on the open book}
Given a random sample $\x_1,\cdots,\x_n$ from distribution $\mu$ on $\M$, the definition, and computation, of EL at a point $\x \in \M$, depends whether $\x$ is on or off the spine. Similarly, asymptotic results for the sample Fr\'echet mean, including the corresponding Wilks' result, depend on the location of the population mean $\x_0$, i.e. whether $\x_0$ is on or off the spine.  

Denote by $\bar \x_n$ the unique sample Fr\'echet mean obtained as the minimizer of \eqref{eq:frechet} with $\mu$ as the empirical measure on observations $\x_1,\cdots,\x_n$. 

\subsection{Away from the spine}
\label{sec:EL_nonspine}
 The empirical likelihood at $\x=(x,k) \in \M$, away from the spine on page $k$, i.e. $x=(x^{(0)}, x^{(1)})$ where $x^{(0)}>0$ and $x^{(1)} \in \mathbb{R}^{p-1}$, is defined using the folding map $F_k$. The characterization \eqref{eqn1b} implies EL for the Fr\'echet mean is given by
%
%
%
\begin{eqnarray}
\log(\mathcal R_n((x,k))):=\max\sum_{i=1}^n\log (n p_i) \qquad\hbox{subject to }\left\{
\begin{array}{l}
\sum\limits_{i=1}^n p_i=1;\,\,p_i\geqslant0;\\
\sum\limits_{i=1}^n p_iF_k(\x_i)=x.
\end{array}
\right.
\label{eqn1d}
\end{eqnarray} 
A feasible solution to problem (\ref{eqn1d}) exists if and only if $\x$ lies in the convex hull of $F_k(\x_1), \ldots , F_k(\x_n)$; if no feasible solution exists then $\log\left ( \mathcal{R}(x,k) \right )=- \infty$.
%



\subsection{On the spine}
\label{sec:EL_spine}
We shall represent $\x$ on the spine $S$ by its $\mathbb R^{p-1}$-coordinate $x^{(1)}$. Then, for an $\x=x^{(1)}$ on $S$ to be the Fr\'echet mean of $\mu$, the characterisation in \eqref{eqn1c} leads to the following formulation of EL:
\begin{eqnarray}
\begin{array}{lcl}
\log(\mathcal R_n(\x)):=\max\sum\limits_{i=1}^n\log (n p_i) ,
\quad \hbox{subject to }\left\{
\begin{array}{l}
\sum\limits_{i=1}^n p_i=1;\,\,p_i\geqslant0;\\
\sum\limits_{i=1}^n p_iP_s(\x_i)=x^{(1)};\\
\sum\limits_{i=1}^np_i\langle F_j(\x_i),e_j\rangle\leqslant0,\quad j=1,\cdots,\ell.
\end{array}
\right.
\end{array}
\label{eqn1e}
\end{eqnarray}
It is convenient to link the above formulation with inequality constraints to an optimisation problem with equality constraints. To this end, let $\{p_1^{[0]}, \ldots , p_n^{[0]}\}$ denote the solution to the following optimisation problem:
\begin{equation}
\log \left (\mathcal{R}_n^{[0]}(\x)  \right ):=\textrm{max} \sum_{i=1}^n \log (n p_i )  \hskip 0.3truein \textrm{subject to} \hskip 0.1truein
\begin{cases} \sum\limits_{i=1}^n p_i=1; \, p_i \geqslant 0;\\
\sum\limits_{i=1}^n p_i P_s(\x_i)=x^{(1)}.
\end{cases}
\label{unconstrained}
\end{equation}
The solution to (\ref{unconstrained}) exists and is unique if $x^{(1)}$ lies in the convex hull of $\{P_s(\x_i): i=1, \ldots , n\}$; otherwise, (\ref{unconstrained}) has no feasible solutions, and the solution to (\ref{unconstrained}) is minus infinity.

For $j=1, \ldots , \ell$, let $\{p_1^{[j]}, \ldots , p_n^{[j]}\}$ be the solution to
\begin{equation}
\log \left ( \mathcal{R}_n^{[j]}(\x)\right):=\textrm{max} \sum_{i=1}^n \log (n p_i ), \quad \textrm{subject to} \hskip 0.1truein
\begin{cases} \sum\limits_{i=1}^n p_i=1; \, p_i \geqslant 0;\\
\sum\limits_{i=1}^n p_i F_j(\x_i)=(0,x^{(1)}).
\end{cases}
\label{constrained_plus_one}
\end{equation}
Note that 
\[
\sum_{i=1}^n p_i F_j(\x_i)=(0,x^{(1)}) 
\]
if and only if the following two conditions are both satisfied:
\[
\sum_{i=1}^n p_i P_s(\x_i) = x^{(1)}, \qquad \sum_{i=1}^n p_i  \langle F_j(\x_i), e_j \rangle = 0;
\]
This observations leads to the following result.
\begin{proposition}[EL on the spine]
\label{prop:ELbook_spine}
Assume that $\x = x^{(1)}$ lies on the spine, $S$.
\begin{enumerate}
\item[(i)] Suppose that $x^{(1)}$ lies in the interior of  the convex hull of $P_s(\x_1), \ldots , P_s(\x_n)$ and that
   \begin{equation}
\sum_{i=1}^n p_i^{[0]} \langle F_j(\x_i), e_j \rangle < 0,\qquad j=1,\ldots,\ell,
\label{test_case}
   \end{equation}
where $ ( p_i^{[0]}  )_{i=1}^n$ is the solution to (\ref{unconstrained}).   Then $ ( p_i^{[0]} )_{i=1}^n$ is also the solution to (\ref{eqn1e}) and therefore  $\log(\mathcal{R}_n(\x))=\log(\mathcal{R}_n^{[0]}(\x))$.\\
\item[(ii)] If equation (\ref{test_case})  fails for some $j=1, \ldots , \ell$, and therefore fails for precisely one such $j$, then  the solution to (\ref{eqn1e}) is given by 
\begin{equation}
\log(\mathcal{R}_n(\x))=\max_{k=1, \ldots , \ell} \log(\mathcal{R}_n^{[k]}(\x)),
\label{max_case}
\end{equation}
where $\log(\mathcal{R}_n^{[j]}(\x))$ is the solution to (\ref{constrained_plus_one}).
\end{enumerate}
\end{proposition}

\vskip 0.1truein

\begin{remark}
Suppose (\ref{test_case}) fails for some $k=1, \ldots , \ell$; note that $k$ is unique if it exists.  However, the solution to (\ref{eqn1e}) will not necessarily satisfy both $\sum_{i=1}^n p_i \langle F_k(\x_i), e_k\rangle=0$ and $\sum_{i=1}^n p_i \langle F_j(\x_i), e_j\rangle <0, j \neq k$. This explains the reason for the formulation of Proposition \ref{prop:ELbook_spine}(ii).
\end{remark}

\subsection{Asymptotic results}
\label{sec:openbook_asymp}
The Fr\'echet mean $\x_0$ of $\mu$ may be non-sticky, half-sticky, or sticky, and the sample Fr\'echet mean $\bar \x_n$ converges with high probability to $\x_0$ as $n \to \infty$ in distinct ways, resulting in three different limiting distributions for its fluctuations about $\x_0$ \citep{HHLMMMNOPS}. In the non-sticky case, its limit distribution is a Gaussian with support on $\mathbb R^p$; in the sticky case, it is asymptotically Gaussian with support on the spine; and, for the half-sticky case,
$\pr(\bar{\x}_n \in S)\to 1/2$ as $n \to \infty$,
and its limit distribution is no longer Gaussian, but still related to one. These different limiting behaviours are also reflected in the limiting distribution of $\mathcal -2 \log (\mathcal R_n(x_0))$ as follows. 

\begin{theorem}[Wilks' theorem on open books]
Let $\x_1,\ldots, \x_n$ be independent random variables with common distribution $\mu$ on $M$ given in (\ref{eq:decomposition}), and assume that the following moment conditions hold:
\begin{equation}
\int_{x^{(1)} \in S} \vert \vert x^{(1)} \vert \vert^2 \d\mu_0(x^{(1)})< \infty, \hskip 0.2truein \int_{x \in L_k^0} \vert \vert x \vert \vert^2 \d \mu_k(x) < \infty, \hskip 0.2truein k=1, \ldots , \ell.
\label{moment_assumptions}
\end{equation}
\begin{enumerate}
\item[(i)] Suppose that the population  Fr\'echet mean $\x_0$ lies in the interior of page $k$, and is non-sticky.  If $F_k(\x_1)$ has a positive definite covariance matrix, then as $n \rightarrow \infty$, $-2 \log \left (\mathcal{R}_n(\x_0)\right )$ converges in distribution to $\chi^2_p$.
\item[(ii)]	Suppose that the Fr\'echet mean $\x_0$ lies in the spine and is sticky.  Then, $-2\log(\mathcal R_n(\x_0))$ converges to $\chi^2_{p-1}$ in distribution as $n\rightarrow\infty$.
\item [(iii)]Finally, suppose that the population  Fr\'echet mean $\x_0$ in half-sticky. Then, as $n\rightarrow\infty$, $-2\log(\mathcal R_n(\x_0))$ converges to $\frac{1}{2}\{\chi^2_p+\chi^2_{p-1}\}$ in distribution. 
\end{enumerate}
\label{thm3}
\end{theorem}

\section{Bootstrap calibration}
\label{sec:bootstrap}
 In many simulation studies of the practical performance of the EL method, it has been seen that inference based on Wilks' theorem tends not to be very accurate unless the sample size is large.  For example, it is typically the case that confidence regions for a  mean in a Euclidean setting tend to undercover when based on the limiting $\chi^2$ limiting distribution.  An effective and  practical method for improving the  inferential accuracy of EL is the use of bootstrap calibration.  The idea of applying bootstrap calibration to EL inference goes back to the earliest days of empirical likelihood \citep{ABO88,ABO}, but we suspect that bootstrap calibration has not been used as often in practice at it might have been.

Here, we consider  bootstrap calibration of EL confidence regions for the Fr\'echet mean $\x_0$ on $M$.  For $\alpha \in (0,1)$, a nominal $100(1-\alpha)\%$ confidence region for $\x_0$ is given by
\begin{equation}
\mathcal{C}_{1-\alpha}=\left \{x \in \M: -2 \log (\mathcal{R}(\x)) \leqslant c_q(\alpha)  \right \},
\label{CR_1}
\end{equation}
where $c_q(\alpha)$ is the tail probability function of the appropriate $\chi^2_q$ distribution such that $\pr[\chi^2_q \geqslant c_q(\alpha)]=\alpha$.   The bootstrap-calibrated confidence region is given by
\begin{equation}
\mathcal{C}^b_{1-\alpha}= \left \{ x \in \M: -2 \log (\mathcal{R}(\x))\leqslant c^b_q(\alpha) \right \},
\label{CR_2}
\end{equation}
where $c^b_q(\alpha)$ satisfies
\begin{equation*}
\pr[-2 \log (\mathcal{R}^\ast (\bar \x_n))  \geqslant \hat{c}_q(\alpha)\vert x_1, \ldots , x_n] \approx  \alpha,
\label{chat}
\end{equation*}
where (\ref{chat})  is approximate  due only to discreteness.  The probability on the LHS of (\ref{chat}) is conditional on the original sample $x_1, \ldots , x_n$ and $\mathcal{R}^\ast (\bar \x_n)$ is the empirical likelihood based on a bootstrap sample $\x_1^\ast, \ldots , \x_n^\ast$, sampled randomly with replacement from the original sample $\x_1, \ldots , \x_n$, and evaluated at the sample Fr\'echet mean $\bar \x_n$ of the original sample.  In practice we estimate $c^b_q(\alpha)$ by sampling $B$ bootstrap samples randomly, with replacement from the original sample, calculate
$u_b=-2 \log(\mathcal{R}^{(b)}\left (\bar \x_n)\right )$, $b=1, \ldots , B$, and then set
\[
c^b_q(\alpha):=u_{(\lfloor B(1-\alpha)\rfloor +1)},
\]
where $u_{(r)}$ is the $r$th smallest order statistic and $\lfloor \cdot \rfloor$ denotes the floor function, i.e. the largest integer less than or equal to the argument of the function.

For $\x \in M$ in a neighbourhood of the Fr\'echet mean$\x_0$, denote by $X_i(\x), i=1, \ldots , n$ the Euclidean random variables obtained by applying any of the folding maps $F_k$ or the projection map $P_s$, depending on the location of $\x \in M$. 
Denote by $X(\x)$ a Euclidean random variable with the same distribution as the $X_i(x)$.  Define the characteristic function of $X(\x_0)$ by
\[
\Psi(t):=E\left [\exp (\iota t^\top X(\x_0))  \right ],
\]
where $\iota$ is the unit imaginary number. Consider the following conditions:
\begin{equation}
\limsup_{\vert \vert t \vert \vert \rightarrow \infty} \vert  \Psi(t) \vert <1;
\label{Condition_1}
\end{equation}
and
\begin{equation}
E\left [\vert \vert X(\x_0) \vert \vert^{10} \right  ] < \infty.
\label{Condition_2}
\end{equation}

We prove the following result, which elucidates on how the bootstrap calibrated confidence region $\mathcal C^b_{1-\alpha}$ for $\x_0$ improves the asymptotic coverage rate of $\mathcal C_{1-\alpha}$ obtained via Wilks' theorem, and then follow with some remarks. 

\begin{theorem} 
\label{Theorem_5.1}
Let  $\x_1, \ldots , \x_n \in \M$ be i.i.d. random variables with distribution $\mu$.  Assume that $\mu$ has a unique Fr\'echet mean $\x_0$.  Suppose also that \eqref{Condition_1} and \eqref{Condition_2} hold and that the real-valued function $\x \mapsto \tau(x)$ defined in Appendix \ref{sec: Def5} has a continuous Hessian in a neighbourhood of $\x_0$.  Then, in either the sticky or non-sticky case,
\begin{equation}
\mathrm{pr}[\mathcal{C}^b_{1-\alpha} \ni \x_0]=1-\alpha +O(n^{-2}),
\label{Theorem_5_statement}
\end{equation}
where $\mathcal{C}^b_{1-\alpha}$ is as defined in (\ref{CR_2}).
\label{thm:bootstrap}
\end{theorem}

\begin{remark}
Bootstrap calibration of EL  has been familiar in the Euclidean setting since the earliest days of EL \citep{ABO88}.  Moreover, it has been stated previously that bootstrap calibration of EL reduces coverage error to $O(n^{-2})$ \citep[e.g., Appendix B][]{FHJW96}.  However,  so far as we are aware, Theorem \ref{thm:bootstrap} is the first statement of such a result under fully explicit sufficient conditions, even in the Euclidean case.  Specifically, the moment condition (\ref{Condition_2}) and the smoothness of the function $\tau$ in \eqref{eq:tau} do not seem to have been stated explicitly before.
\end{remark}
\begin{remark}
If the moment condition (\ref{Condition_2})  is weakened to $\mathbb{E}[\vert \vert X(\x_0) \vert \vert^6] <\infty$ and all the other conditions hold with the possible exception of the smoothness assumption on $\tau (\x)$, which is now not needed, then \eqref{Theorem_5_statement} holds with $O(n^{-1})$ replacing $O(n^{-2})$.
\end{remark}

\begin{remark}
Bootstrap calibration may also be applied to EL-based hypothesis testing, e.g. $k$-sample problems for Fr\'echet means.  In this case, if each of the $k$ populations satisfies the assumptions of Theorem \ref{thm:bootstrap}, then the bootstrap-calibrated  test will have actual size differing from the nominal size by $O(n^{-2})$, as opposed to an $O(n^{-1})$ error in the case of a test based on Wilks' theorem.  See \cite{AMWO} for discussion of  bootstrap calibration for EL for the extrinsic Fr\'echet mean on the complex projective. 
\end{remark}

\begin{remark}
Asymptotically, Bartlett correction of EL and bootstrap calibration of EL produce confidence regions with the same coverage error order $O(n^{-2})$.  Our expectation is that bootstrap calibration will do considerably better than Bartlett correction for small-to-moderate values of $n$.  We do not provide any simulation evidence here, but see the discussion on p. 34 of \cite{ABO}.
\end{remark}

\begin{remark}
In the semi-sticky case, because of the mixture structure of the limit distribution given in Theorem (\ref{thm3}), part (iii), the coverage error is of size $n^{-1}$ as opposed to size $n^{-2}$.
\end{remark}

\section{Numerical illustrations on the 3-spider}
\label{sec:numerics}
\subsection{Algorithm for EL calculation}
We investigate utility of the respective Wilks' results for inference on the population Fr\'echet mean for the $\ell$-spider with $\ell=3$, a special case of the open book with $p=1$. Specifically, we examine rejection rates for the point null hypothesis $H_0:\x_0=\z$ and confidence regions constructed using both $\chi^2$ and bootstrap calibrated percentile points. 

Bootstrap calibration of the Wilks' statistic to test for Type I error is carried out in the usual manner by comparing  $-2\log \mathcal R_n(\z)$ to the empirical quantile from the sample $\{W^*_{1}, W^*_{2}, \ldots, W^*_{B}\}$ with $W^*_{b}:=-2\log \mathcal R_n(\bar \x^{(b)}_N), b=1,\ldots,B$, where $\bar \x^{(b)}_N$ is the sample Fr\'{e}chet mean of the $b$th bootstrap sample of size $N$ obtained by sampling with replacement. 

Let $\M$ be the 3-spider: the open book with three pages with $p=1$, sometimes, quite appropriately, referred to as the spider with three `legs'; in this particular context  we use `legs'. The three legs $L_k=\mathbb R_{\geqslant 0}$ with interiors $L^o_k=\mathbb R_{>0}$ for $k=1,2,3$ are joined together at the spine $\{0\}$. Since $p=1$, the second coordinate of $x=(x^{(0)},x^{(1)})$ in $\x=(x,k)$ is not required; we simply use $x$ for $x^{(0)}$. The folding map $F_k$ for page $k$ assumes the simple form
\[
F_k((x,j))=
\begin{cases}
\quad x&j=k;\\
-x & j \neq k.
\end{cases}
\]
The distance between two points $\x_1=(x_1,j)$ and $\x_2=(x_2,k)$ with $x_1>0,x_2>0 $ and $j,k \in\{1,2,3\}$ is then
\[
d(\x_1,\x_2)=
\begin{cases}
|x_1-x_2| & \text{if } j=k;\\
|x_1+x_2| & \text{if } j \neq k.
\end{cases}
\]

\vskip 6pt
Computation of $\log \mathcal R_n(\x)$ depends on whether $\x$ is on the spine $\{0\}$. When $\x \in L^0_k, k=1,2,3$ away from the spine, the optimisation in \eqref{eqn1d} reduces to
\begin{eqnarray}
\begin{array}{lcl}
&&\log(\mathcal R_n((x,k))=\max\sum\limits_{i=1}^n\log (n p_i),\\
&&\hbox{subject to }\left\{
\begin{array}{l}
\sum\limits_{i=1}^n p_i=1;\,\,p_i\geqslant0;\\
\sum\limits_{i=1}^np_i F_k(\x_i)=x,\quad \text{for } \x \in L^0_k.
\end{array}
\right.
\end{array}
\label{eqn3}
\end{eqnarray}
When $\x$ is on the spine, the optimisation in \eqref{eqn1e} reduces to computing
\begin{eqnarray}
\begin{array}{lcl}
&&\log(\mathcal R_n(\x))=\max\sum\limits_{i=1}^n\log (n p_i),\\
&&\hbox{subject to }\left\{
\begin{array}{l}
\sum\limits_{i=1}^n p_i=1;\,\,p_i\geqslant0;\\
\sum\limits_{i=1}^np_i F_k(\x_i)\leqslant 0,\quad k=1,2,3.
\end{array}
\right.
\end{array}
\label{eqn3a}
\end{eqnarray}
As discussed earlier with the open book, computation of \eqref{eqn3a} can be simplified by considering an equivalent formulation with two cases based on the observations $P_s(\x_1),\ldots,P_s(\x_n)$ described in \eqref{unconstrained} and \eqref{constrained_plus_one}, and summarized in Proposition \ref{prop:ELbook_spine}. However, for the $3$-spider it easier to work directly with \eqref{eqn3a}, characterized by the following two cases pertaining to behaviour of sample means $\bar \x_n^k:=(\bar x^k_n,k)=\frac{1}{n}\sum_{i=1}^n F_k(\x_i)$ of the Euclidean random variables $F_k(\x_1),\ldots,F_k(\x_n), k=1,2,3$  \citep[Lemma 5.2.2]{CharlesThesis}.

\begin{enumerate}[leftmargin=*]
 \item[(i)] When there exists a $k \in \{1,2,3\}$ such that  $\sum_{i=1}^n F_k(\x_i) \geq 0$, note that $\bar x_n^k \geqslant 0$ but $\bar x_n^j<0$ for $j \leqslant k$, and the optimisation in \eqref{eqn3a} is equivalent to one with the constraint $\sum_{i=1}^n p_iF_k(\x_i)=0$, in addition to $p_i \geqslant 0$ and $\sum_ip_i=1$. 
\item[(ii)] When for every $k=1,2,3$, $\sum_{i=1}^n F_k(\x_i)<0$, note that the  $\bar x_n^k<0, k=1,2,3$, and the optimisation in \eqref{eqn3a} is equivalent to one without an additional constraint to $p_i \geqslant 0$ and $\sum_ip_i=1$, which has the solution $p_1=\cdots=p_n=1/n$. The solution clearly satisfies $ \sum_ip_iF_k(\x_i) \leqslant 0$. 
\end{enumerate}
Algorithm \ref{algo1} summarizes computation of the EL for the 3-spider.


\begin{algorithm}[!htb]
\caption{: EL computation on the 3-spider}
\begin{algorithmic}[1]
\State \textbf{Input}: Data $\x_1,\ldots,\x_n$ and point $\x$.
\State \textbf{Output}: $\log \mathcal R_n(\x)$.
\State Compute the leg-dependent sample Euclidean means $\bar x_n^k=\frac{1}{n}\sum_i F_k(\x_i), k=1,2,3$.
\If {$\x$ lies on $L^0_k, k=1,2,3$ away from the spine},
\State compute $\log \mathcal R_n(\x)$ using \eqref{eqn3}.
\Else { When $\x$ lies on the spine},
\If{$\bar x_n^k<0, k=1,2,3$}
\State $\log \mathcal R_n(\x)=1$;
\If{there exists a $k$ such that $\bar x_n^k \geqslant 0$ and $\bar x_n^j<0, j\neq k$}
\State $\log \mathcal R_n(\x)=\log \mathcal R_n(0)$, under the constraint $\sum_{i=1}^n p_iF_k(\x_i)=0$.
\EndIf
\EndIf
\EndIf
\end{algorithmic}
\label{algo1}
\end{algorithm}
As  $n \to \infty$, Wilks' result in Theorem \ref{thm3} then implies that $-2\log \mathcal R_n(\x_0)$ converges in distribution: (i) 0 for the sticky case; (ii) $\chi^2(1)$ for the non-sticky case; and (iii) to the distribution function $H$ on the line given by
\[
H(x)=
\begin{cases}
0&x<0;\\
\frac{1}{2}&x=0;\\
\frac{1}{2}+\frac{1}{2}\thinspace G(x)&x>0, 
\end{cases}
\]
for the partially sticky case, where $G$ is the distribution function of a $\chi^2_1$ random variable. 
\subsection{Simulated data}
We consider four simulation settings to examine performance of Algorithm 1 corresponding the nature of the Fr\'{e}chet mean. The distribution $\mu$ in \eqref{eq:decomposition} is taken to be 
\begin{equation}
	\label{eq:mix_exp}
\mu=w_1E(a_1)+w_2E(a_2)+w_3E(a_3),
\end{equation}
 where $E(a)$ denotes the distribution function of an exponential random variable with mean parameter $a$ and $w_i\in [0,1]$ is the restriction of $\mu$ to leg $i$, with $w_1+w_2+w_3=1$. We generate $n$ points $\x_1,\ldots,\x_n$ with $n=100$ on the 3-spider under the following settings.  
\begin{enumerate}[leftmargin=*]
\itemsep 0em
\item[(a)] \emph{Non-sticky case}. $w_1=1/2,w_2=1/3, w_3=1/6$ and $a_1=1, a_2=2, a_3=1$ with true Fr\'{e}chet mean at 1/6 on leg 1. 
\item[(b)] \emph{Half-sticky case with mean on the spine}. $w_1=1/2,w_2=1/4, w_3=1/4$ and $a_1=1, a_2=1, a_3=1$ with true Fr\'{e}chet mean on the spine. 
\item[(c)] \emph{Half-sticky case with mean off the spine}. $w_1=1/2,w_2=1/4, w_3=1/4$ and $a_1=1, a_2=1, a_3=1$ with true Fr\'{e}chet mean off the spine on leg 1 at 1/6. 
\item[(d)] \emph{Sticky case}. $w_1=1/3,w_2=1/3, w_3=1/3$ and $a_1=1, a_2=1, a_3=1$ with true Fr\'{e}chet mean on the spine. 

\end{enumerate}

Next, we compute rejection rates of the Wilks' test based on the $\chi^2$ distribution for a hypothesis test for the Fr\'echet mean $\x_0$ of $\mu$ under the null and the alternative hypotheses
\[
H_0: \x_0=\bm z \quad vs \quad H_1: \x_0 = \bm z_1 (\neq \bm z).
\]

Observations are generated by sampling from $\mu$ in \eqref{eq:mix_exp} with 
\[
(w_1,w_2,w_3)^\top=(1/2,1/4,1/4)^\top, \hskip 0.2truein (a_1,a_2,a_3)^\top=(1/5,2,2)^\top.
\]
The true Fr\'echet mean is non-sticky and located at 9/4 on leg 1. Under the asymptotic $\chi^2_1$ distribution Type I error probabilities are computed from 500 Monte Carlo runs with $\bm z=9/4$. Table \ref{table1} records the rejection probabilities of the Wilks'-based test for different sample sizes $n$ for significance level $\alpha=0.05$ and the corresponding bootstrap calibrated Type I error probabilities with bootstrap sample size $N=500$ within each Monte Carlo run. We observe that for small sample sizes the bootstrap calibrated version of the Wilks' statistic has lower Type I error probability when compared to the test based on the $\chi^2$ percentile, but their performances are quite similar for larger sample sizes. 

When computing the Type II error probabilities, we considered a mean close to that under $H_0$ (10/4) and another farther (13/4) with the same weights $(w_1,w_2,w_3)^\top=(1/2,1/4,1/4)^\top$  but parameters $(2/11, 2, 2)^\top$ and $(2/14, 2, 2)^\top$, respectively, for the parameters of the exponential distributions of the mixture. We see in Table \ref{table1} that for relatively small sample sizes (e.g., 10,20) since the probability of Type I error is large, the corresponding Type II error probability is also relatively large; this phenomenon decreases with increasing sampling size and increasing distance between $\bm z_1$  and $\bm z$. 
\begin{table}[!htb]
\begin{center}
{\small
\begin{tabular}{c|c|c|c}
Error&Sample size $n$ & $\chi^2$ & Bootstrap\\
\hline
\multirow{4}{*}{\parbox{2cm}{Type I \\ $\bm z=9/4$ \\ on leg 1}}	&10 & 0.1530 (0.1138)&0.0695 (0.0804)\\
	&20 & 0.1005 (0.0672)&0.0690 (0.0567)\\
	&50 & 0.0595 (0.0335) & 0.0545 (0.0321)\\
	&200 & 0.0670 (0.0177) & 0.0630 (0.0172)\\
	\hline

	\multirow{4}{*}{\parbox{2cm}{Type II \\ $\bm z_1=10/4$ \\ on leg 1}}	&10 & 0.842 (0.115)&1.000 (0.000)\\
	&20 & 0.880 (0.073)&1.000 (0.000)\\
	&50 & 0.914 (0.040) &1.000 (0.000)\\
	&200 & 0.892 (0.022)& 1.000 (0.000)\\
	\hline
	\multirow{4}{*}{\parbox{2cm}{Type II \\ $\bm z_1=13/4$ \\ on leg 1}}	&10 & 0.830 (0.119)&1.000 (0.000)\\
	&20 & 0.828 (0.084)&1.000 (0.000)\\
	&50 & 0.742 (0.062) & 0.984 (0.018)\\
	&200 & 0.244 (0.030) & 0.308 (0.033)\\
\end{tabular}
\caption{Type I and II errors (for significance level $\alpha=0.05$) from 500 Monte Carlo runs of the Wilk's test in Theorem \ref{thm3} with asymptotic distribution $\chi^2_1$, and the corresponding boostrap calibrated test with bootstrap sample size $N=500$ in the non-sticky case. }
\label{table1}
}
\end{center}
\end{table}

\subsection{Phylogenetic data}
The BHV treespace proposed by \cite{BHV} as a geometric setting for studying phylogenetic trees on a fixed number of leaves admits an open book decomposition. Thus the space of phylogenetic trees on three leaves with at most one internal edge can be identified with the 3-spider $\M$. We consider the metazoan data in \cite{TN}, originally analysed by \cite{KHK}. 
\vskip 6pt
We extract 3-leaf phylogenetic trees corresponding to the species ``Calb", ``Scas" and ``Sklu" since they appear together in 99 out of the 106 trees in the dataset. The three legs of the spider represent the three possible tree topologies as shown in Figure \ref{fig2}. A point $\x=(x,k)$ on the 3-spider represents a phylogenetic tree with topology indexed by leg $k$ with internal branch length $x>0$. Shrinking $x$ to 0 results in a tree with no internal edge; all such trees reside on the spine $\{0\}$. We suppose that the first leg $L_1$ represents the first topology in Figure \ref{fig2}, $L_2$ the middle topology, and $L_3$ the last. 
\begin{figure}[!htb]
\centering
\includegraphics[scale=0.3]{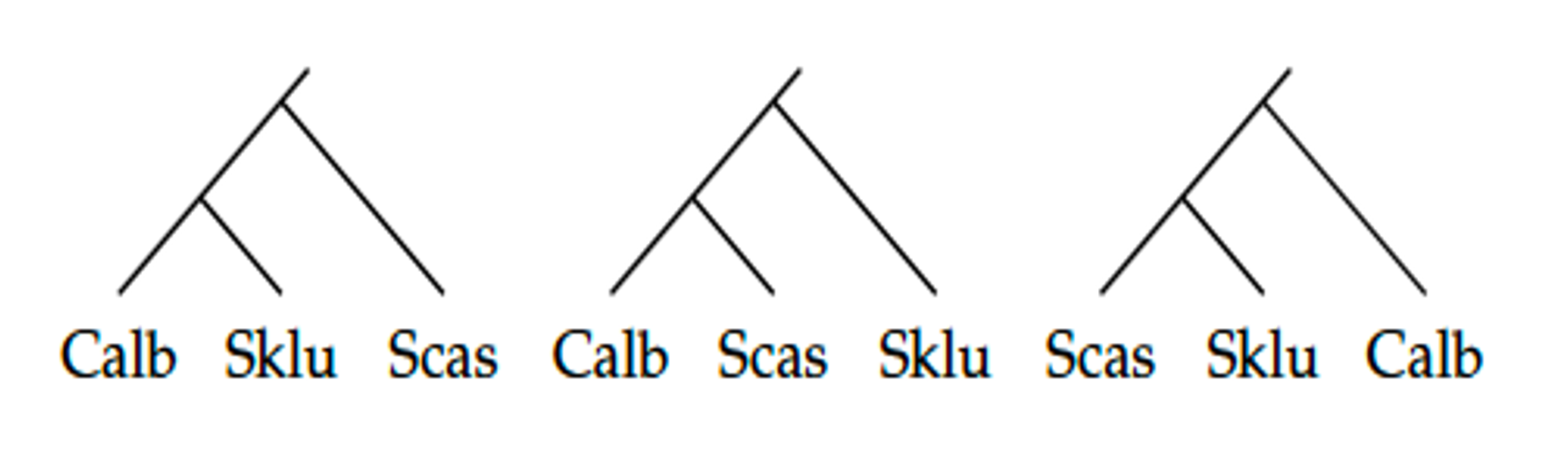}
\caption{Three possible topologies for data extracted from metazoan data \citep{TN}.}
\label{fig2}
\end{figure}

We consider constructing confidence sets for the population Fr\'echet mean phylogenetic tree using Wilks' result. Discontinuity of $\log \mathcal R_n(\x)$ at the spine $\{0\}$ ensures that for a fixed confidence level $\alpha$ there are several choices for such a construction, since there are three copies of the same space $\mathbb R_{>0}$.  
\vskip 6pt
Figure \ref{fig3} illustrates the potential non-standard nature of confidence set $\{\x \in \M: -2\log \mathcal R_n(\x) \leqslant c_1(1-\alpha)\}$ for the Fr\'{e}chet mean using Theorem \ref{thm3} even in the non-sticky case, where $c_1(1-\alpha)$ is the $100(1-\alpha)$th percentile of the $\chi^2_1$ distribution. The three plots are of $\log \mathcal R_n(\x)$ computed using the folding maps when $\x$ lies in interior of the three legs, away from the spine. We note that the maximum of zero is attained on leg $L_1$; in other words, the EL estimator possesses the topology of the left most tree in Figure \ref{fig2} with internal branch length $\approx 0.08$. 

The red, blue and green lines, respectively, are values of $-c_1((1-\alpha)/2)$ at $\alpha=0.1,0.05, 0.01$. The purple, brown and orange lines are the values of $\log \mathcal R_x(0)$ at the spine under the folding maps $F_1, F_2$ and $F_3$, respectively; these values correspond to percentiles  $-c_1((1-\alpha)/2)$ at $\alpha \approx 0.0001, 0, 0$, respectively. 

\begin{figure}[!htb]
\centering
\includegraphics[scale=0.4]{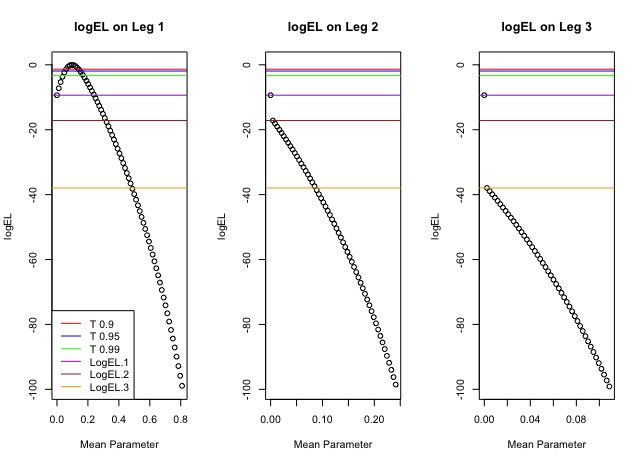}
\caption{$\log \mathcal R_n(\x)$ for metazoan data. When $\x$ is on leg 1 Red, blue and green lines are values of $-c_1((1-\alpha)/2)$ at $\alpha=0.1,0.05, 0.001$, respectively. The purple, brown and orange lines are $\alpha=0.99998,1,1$, respectively.}
\label{fig3}
\end{figure}

From the behaviour of $\log \mathcal R_n(\x)$ in Figure \ref{fig3}, we first present a method to construct confidence sets for the non-sticky mean. With $\alpha_1=0.0001, \alpha_2 \approx 0, \alpha_3 \approx 0$, depending on where the chosen confidence level $\alpha$ lies relative to $\alpha_1,\alpha_2,\alpha_3$, topology of the confidence set can change. 
\begin{enumerate}[leftmargin=*]
\itemsep 0em
\item[(i)] when $\alpha \leqslant \alpha_1$ the confidence set does not include the spine and is a standard interval; 
\item[(ii)] when $\alpha_1 < \alpha \leqslant \alpha_2$, one end point of the confidence set is the spine; 
\item[(iii)] when $\alpha_2< \alpha \leqslant \alpha_3$ the confidence set splits into two parts, with a segment on leg $L_1$ and another on leg $L_2$; 
\item[(iv)] when $\alpha>\alpha_3$ the confidence set comprizes three parts, one corresponding to each leg. 
\end{enumerate}
Figure \ref{fig4} illustrates this phenomenon, where we note the non-standard nature of the confidence sets for cases (ii)-(iv). 
\begin{figure}[!htb]
\centering
\includegraphics[scale=0.4]{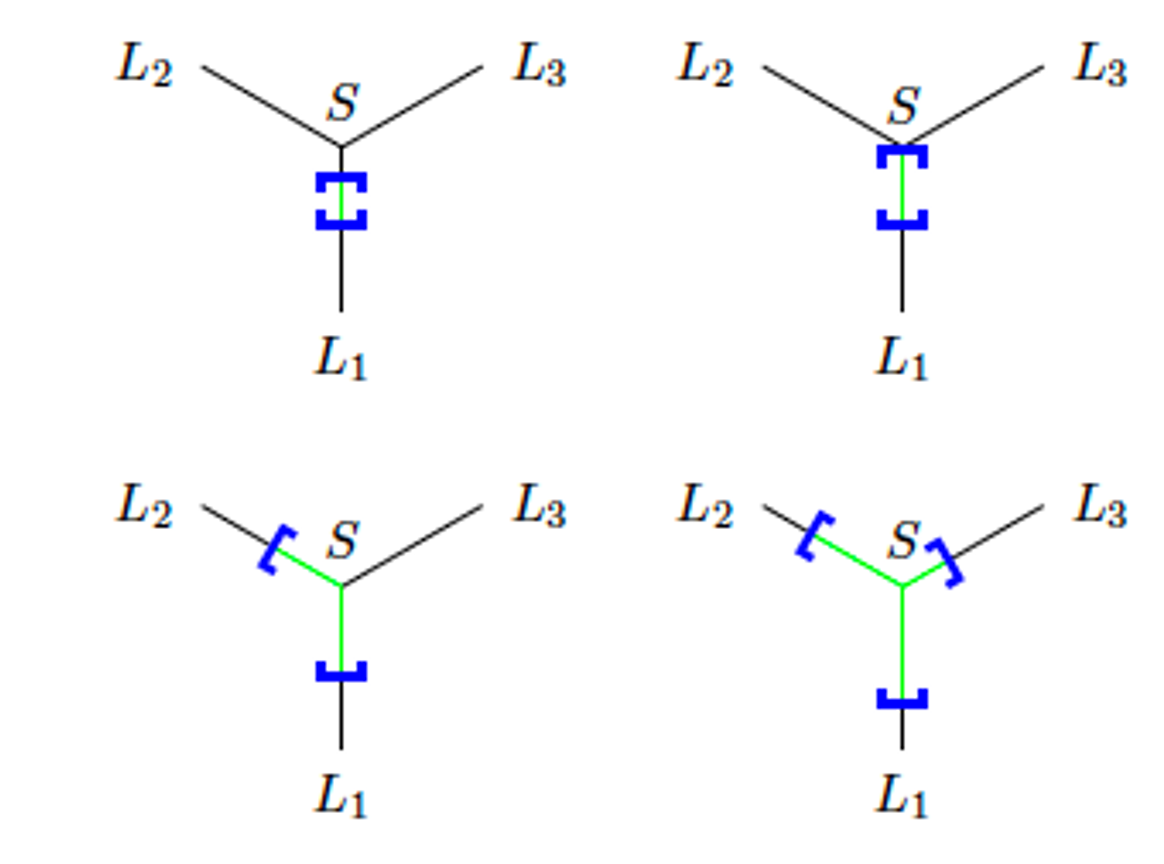}
\caption{Confidence sets for the non-sticky Fr\'{e}chet mean tree of the metazoan data depending on confidence level $\alpha$. Clockwise: (i) when  $\alpha \leqslant \alpha_1$; (ii) when $\alpha_1 < \alpha \leqslant \alpha_2$; (iii) when when $\alpha_2< \alpha \leqslant \alpha_3$; (iv) when $\alpha>\alpha_3$.}
\label{fig4}
\end{figure}



\section*{Acknowledgements}
	KB acknowledges support from grants the Engineering and Physical Sciences Research Council (EPSRC), National Science Foundation (NSF), and the National Institutes of Health (NIH). ATAW acknowledges support from Australian Research Council grant Discovery Project DP220102232.

\begin{appendix}
\section{Definitions and properties}
\label{sec:defs}
\subsection{CAT(0) spaces}
\label{sec:app_cat}
Let $(M,d)$ be a metric space. A path in $M$ is a continuous map $\gamma:[0,1] \to M$ whose length  $L(\gamma)$ is defined by
\[
L(\gamma):=\sup \sum_{i=0}^{n-1}d(\gamma(t_i),\gamma(t_{i-1})),
\]
where the supremum is taken over all $n \geq 1$ and all $0\leq t_0 \leq \cdots\leq t_n\leq1$. Given $x,y \in M$ a path $\gamma:[0,1] \to M$ connects $x$ and $y$ if $\gamma(0)=x$ and $\gamma(1)=y$. $M$ is a \emph{length space} if for all $x,y \in M$, 
\begin{equation}
	\label{geodesic}
	d(x,y)=\inf_{\gamma}L(\gamma),
\end{equation}
where the infimum is taken over all paths $\gamma$ that connect $x$ and $y$.  A length space is said to be a \emph{geodesic space} if for all $x,y \in M$, the infimum in \eqref{geodesic} is attained. 

The notion of curvature of a geodesic metric space $M$ is defined through a comparison with model spaces $C(\kappa)$ with constant curvature $\kappa \in \mathbb R$. The model spaces $C(\kappa)$ are the 2-sphere for $\kappa>0$, the plane $\mathbb R^2$ for $\kappa=0$, and the two-dimensional hyperbolic space for $\kappa <0$. A curvature value is ascribed to $M$ by comparing a geodesic triangle in $M$ with a triangles in $C(\kappa)$. A geodesic triangle $\Delta(a,b,c)$ in $M$ consists of three points $a,b,c$ and geodesic segments connecting them.  Consider a triangle on the plane $C(0)$ made from points $a',b'c'$ with lengths of sides matching the corresponding ones made by $a,b,c$ in $\Delta(a,b,c)$. The geodesic metric space $M$ is said to be a CAT(0) space if every triangle $\Delta(a,b,c)$ is `thinner' than its counterpart in $C(0)$ in the sense that
\[
d(x,a) \leq \|x'-a'\|
\]
for any $x$ on the geodesic segment between $b$ and $c$, where $x'$ is the unique point on the edge $b'c'$ such that $d(b,x)=\|b'-x'\|$ and $d(x,c)=\|x'-c'\|$. A geodesic metric space $M$ is said to have global nonpositive curvature if it is CAT(0). Two important properties of CAT(0) spaces are the following. 
\begin{proposition}
	Let $(M,d)$ be a CAT(0) space. 
	\begin{enumerate}
		\item There is a unique geodesic between every pair of points \citep[Proposition II 1.4][]{BH}. 
		\item Let $\mathcal P$ be the set of all probability measures on $M$ satisfying 
        \[
        \int_M d^2(x,y)p(\d y) <\infty, \quad p \in \mathcal P.          
        \]
        The Fr\'echet means of probability measures in $\mathcal P$ are unique \citep[Proposition 4.3][]{KTT}. 
	\end{enumerate}
\end{proposition}

\subsection{Definitions related to Theorem \ref{thm:bootstrap}}
\label{sec: Def5}
Here we define the real-valued function $\x \mapsto \tau(\x)$ in a neighbourhood of the Fr\'echet mean $\x_0 \in \M$ used in the statement of Theorem \ref{thm:bootstrap}.  First, we recall some notation, slightly modified for present purposes,  from \cite{DHR91}. Denote by $X^j(\x)$ the $j$th element of $X(\x)$. Define
\[
\alpha^{m_1 \cdots m_k}(x)=E[X_1^{m_1}(x) \cdots X_1^{m_k}(x)], 
\]
\[
A^{m_1 \cdots m_k}(x) = n^{-1} \sum_{i=1}^n \left \{ X_i^{m_1}(x) \cdots X_i^{m_k}(x)- \alpha^{m_1 \cdots m_k}(x)  \right \},
\]
\[
\bar{\alpha}^{j_1 \cdots j_k}(x) = \alpha_\ast^{j_1 m_1}(x) \cdots \alpha_\ast^{j_k m_k}(x) \alpha^{m_1 \cdots m_k}(x),
\]
\[
\bar{A}^{j_1 \cdots j_k}(x) = \alpha_\ast^{j_1 m_1}(x) \cdots \alpha_\ast^{j_k m_k}(x) A^{m_1 \cdots m_k}(x),
\]
where a summation convention is used above (a repeated superscript implies summation over that superscript) and $\alpha_\ast^{jk}(x)$ is the inverse of $\alpha^{k\ell}(x)$, i.e.
\[
\alpha_\ast^{jk}(x) \alpha^{k \ell}(x)= \delta^{j\ell}=\begin{cases} 1 &j=\ell \\ 0 & j \neq \ell.  \end{cases}
\]
The function $\tau: M \to \mathbb R$ is defined as
\begin{equation}
\x \mapsto	\tau(\x)=\frac{5}{3} \bar{\alpha}^{jk\ell}(\x)\bar{\alpha}^{jk\ell}(x)-2\bar{\alpha}^{jjk}(\x)\bar{\alpha}^{k \ell \ell}(\x) +\frac{1}{2} \bar{\alpha}^{j j k k}(x).
\label{eq:tau}
\end{equation}

\section{Proofs}
\label{sec:proofs}
\subsection{Proof of Theorem 2.2}
We begin with what is surely a well-known result (stated in slightly more elaborate form  than is needed here).  We state and prove it here as we do not know of a convenient reference.  Denote by $\chi_m^2(a)$ the noncentral chi-squared with $m$ degrees of freedom and noncentrality parameter $a$, with distribution function
\begin{equation}
F_{m,a}(y) = \sum_{k=0}^\infty e^{-a/2}\frac{(a/2)^k}{k!}
F_{m+2k, 0}(y).
\label{mixture_representation}
\end{equation}

\begin{lemma}
The noncentral chi-squared is stochastically monotone without bound in both $m$ and $a$ in the sense that, for all $y\geq 0$, $0 \leq a \leq b$ and $0 \leq m \leq n$,
\begin{equation}
F_{m,a}(y) \geq F_{n,b}(y);
\label{chi_1}
\end{equation}
for any fixed $a \geq 0$,
\begin{equation}
\lim_{m \rightarrow \infty} F_{m,a}(y)=0;
\label{chi_2}
\end{equation}
and, for all $m \geq 0$,
\begin{equation}
\lim_{a \rightarrow \infty} F_{m,a}(y) = 0.
\label{chi_3}
\end{equation}
\end{lemma}

\begin{proof}
Result (\ref{chi_1}) follows directly from the fact that, if $m$, $a$, $h$ and $c$ are all non-negative and $\chi_m^2(a)$ and $\chi_h^2(c)$ are independent,  then $\chi_m^2(a)+\chi_h^2(c) = \chi_{m+h}^2(a+c)$ in distribution.  This result is easily proved using the moment generating function of the noncentral chi-squared.  Result (\ref{chi_2}) follows from the additive property of the noncentral chi-squared mentioned above while result (\ref{chi_3}) follows from the Poisson mixture representation (\ref{mixture_representation}) and the fact that the Poisson distribution is stochastically increasing without bound in its mean parameter.
\end{proof}

We now prove Theorem 2.

\begin{proof}
Consider the EL ratio $\mathcal{R}_n(x)$ for a multivariate Euclidean mean based on a sample $x_1, \ldots , x_n$.  It is well-known that $\log \left (\mathcal{R}_n(x) \right  )$ is a concave function of $x$ \citep[e.g.][Theorem 3.2]{CharlesThesis}.
 It turns out that, under the assumptions of the Theorem, for any $a \in \mathbb{R}^p$,
\begin{equation}
-\log \left ( \mathcal{R}_n(x_0+n^{-1/2} a) \right ) \rightarrow \chi_p^2\left ( a^\top \Sigma_0^{-1}a  \right)
\label{local_alternative_result}
\end{equation}
in distribution, under the assumption that the covariance matrix $\Sigma_0$ of $x_1$ has full rank $p$.  This result for local alternatives is mentioned by \citep[eq. (2.7)]{ABO} in the univariate case; a  detailed proof in the multivariate case under the conditions of the Theorem is given in \citep[Theorem 3.1][]{CharlesThesis}. As a by-product of the proof of the local alternative result, it is proved that $\log (\mathcal{R}_n(x_0+n^{-1/2} v))$ is asymptotically locally quadratic in the sense that
\begin{equation}
\log \left ( \mathcal{R}_n (x_0+n^{-1/2} a)     \right )
= - (z_n - a)^\top \Sigma_0^{-1} (z_n -a) + \delta_n(a),
\label{locally_quadratic}
\end{equation}
where $z_n = n^{1/2}(\bar{x}_n-x_0)$ and, for any fixed $\rho>0$, 
\[
\sup_{v \in B_\rho(0)} \vert \delta_n(v) \vert = o_p(1),
\]
with $B_\rho(0) \subset \mathbb{R}^p$ denoting the ball of radius $\rho$ centred at the zero vector.  A further important point is that, due to $\log (\mathcal{R}_n(x))$ being a concave function with  stationary maximum of 0 at $\bar{x}_n$, the sample mean of the $x_i$, for any $v \in \mathbb{R}^p$ and any $0 \leq t \leq s <\infty$, we have
\begin{equation}
\log \left \{ \mathcal{R}_n \left (\bar{x}_n+t (v -\bar{x}_n) \right ) \right \} \geq \log \left \{ \mathcal{R}_n \left (\bar{x}_n+ s (v-\bar{x_n})   \right ) \right \}
\label{monotonic}
\end{equation}
i.e. $\log ( \mathcal{R}_n(x))$ is a monotonic non-increasing function of $t \geq 0$.  Moreover, at $t=0$, (\ref{monotonic}) is zero and at $t=1$ (\ref{monotonic}) is equal to $\log ( \mathcal{R}_n(v))$.

Our goal is to prove that for any fixed simple alternative $x \in \mathbb{R}^p$ with $x \neq x_0$, $\log (\mathcal{R}_n(x)) \rightarrow -\infty$ in distribution.  We now present the key steps in the argument.

Fix $\epsilon >0$ to be arbitrarily small and fix  $A \in (0,\infty)$ to be arbitrarily large.  First, using (\ref{chi_3}) and (\ref{local_alternative_result}), we choose $a=a(A,\epsilon)$ to have norm $\vert \vert a \vert \vert$ sufficiently large so that
\begin{equation}
\lim_{n \rightarrow \infty} \pr\left [ -\log \left ( \mathcal{R}_n(x_0+n^{-1/2} a) \right ) < A\right  ]=\pr\left [\chi_p^2 \left (  a^\top \Sigma_0^{-1} a \right )<A \right ] <\epsilon/2.
\label{first_of_three}
\end{equation}

Next, using the law of large numbers and recalling the definition $z_n=n^{1/2}(\bar{x}_n-x_0)$, choose $\xi=\xi(\epsilon)>0$ so large  that 
\begin{equation}
\lim_{n \rightarrow \infty} \pr\left [ z_n \notin B_\xi(0)  \right] < \epsilon/2.
\label{second_of_three}
\end{equation}

Third, fix any $\alpha>1$ and define
\begin{equation}
t_0 (z)=\sqrt{\frac{(z-a)^\top \Sigma_0^{-1} (z-a)}{(x-x_0)^\top \Sigma_0^{-1} (x - x_0)}}.
\label{def_t_0}
\end{equation}
Note that, when $a$ and $\Sigma_0$ are fixed and $\Sigma_0$ has full rank $p$, 
\begin{equation}
0 \leq \sup_{z \in B_\xi (0)}\vert t_0(z) \vert \leq t_0^\ast <\infty.
\label{t_star}
\end{equation}
We claim that on the event $\{z_n \in B_\xi(0)\}$, and for $n=n(a, \alpha, \xi)$ sufficiently large, 
\begin{equation}
\log \left ( \mathcal{R}_n(x_0+n^{-1/2}a) \right )
\geq \log \left \{\mathcal{R}_n \left (\bar{x}_n +n^{-1/2} \alpha t_0(z_n) (x -\bar{x}_n ) \right )  \right \}
\label{third_of_three}
\end{equation}
  To justify (\ref{third_of_three}), we first use the locally quadratic representation (\ref{locally_quadratic}) to write
\begin{align}
\log \left \{  \mathcal{R}_n \left ( \bar x_n +n^{-1/2} t (x -\bar x_n)  \right )   \right \} 
&=\log \left \{  \mathcal{R}_n \left ( x_0 +n^{-1/2} z_n +n^{-1/2} t (x-\bar x_n)  \right )   \right \} \nonumber \\
&=\left ( z_n - z_n -t(x-\bar x_n)    \right )^\top \Sigma_0^{-1} \left ( z_n - z_n -t(x-\bar x_n) \right )
\nonumber \\
& \hskip 1.5truein + \delta_n\left ( v(z_n,,t) \right )
\nonumber \\
&=t^2 (x_1-\bar x_n)^\top \Sigma_0^{-1}(x - \bar x_n)
+  \delta_n\left ( v(z_n,,t) \right ),
\label{one_other_thing}
\end{align}
where 
\[
v(z_n,t)=z_n+t(x -\bar x_n) = (1-t) \bar x_n+ t y.
\]
An important point is that, with $t_0^\ast$ defined in (\ref{t_star}),  
\begin{equation}
\max \left \{ \vert \delta_n(a) \vert, \sup_{z \in B_\xi (0)}     \sup_{t \in [0,t_0^\ast]}\vert \delta_n \left \{ v(z, \alpha t \right ) \vert  \right \} = o_p(1),
\label{uniform_zero}
\end{equation}
because this tells us that the collection of remainder term $\delta_n \{v(z_n,\alpha t)$, in the locally quadratic expansions (\ref{locally_quadratic}) that we need to consider,  are going to zero uniformly in probability.  For this reason we do not need to discuss these remainder terms explicitly below, as they play a negligible role.
Consequently, starting with with (\ref{one_other_thing}) with $t=\alpha t_0(z_n)$ and for  $n$ sufficiently large, 
\begin{align*}
\alpha^2  & \left ( \frac{(z_n-a)^\top \Sigma_0^{-1} (z_n-a)}
{(x - x_0)^\top \Sigma_0^{-1} (x - x_0)}  \right )
(x-\bar x_n)^\top \Sigma_0^{-1}(x - \bar x_n)\\
&\geq \alpha^2  \left (\frac{(y-\bar x_n)^\top \Sigma_0^{-1}(x - \bar x_n)}{(x - x_0)^\top \Sigma_0^{-1} (x - x_0)} \right ) (z_n-a)^\top \Sigma_0^{-1} (z_n-a)\\
&\geq (z_n-a)^\top \Sigma_0^{-1} (z_n-a)
\end{align*}
because $\alpha>1$ and, on the set $\{z_n \in B_\xi(0)\}$ and for $n$ sufficiently large, $\bar x_n = x_0+n^{-1/2}z_n=x_0+ O(n^{-1/2})$.  Therefore, bearing in mind the minus sign in (\ref{locally_quadratic}), the monotonicity in (\ref{monotonic})
and the uniform convergence to zero in distribution indicated in (\ref{uniform_zero}),
\begin{align}
-\log \left \{ \mathcal{R}_n \left ( x_0 + n^{-1/2} a \right )   \right \}
& \leq -\log \left \{ \mathcal{R}_n \left (\bar x_n +n^{-1/2} \alpha t_0(z_n) (x-\bar x_n) \right )\right \} \nonumber \\
& \leq - \log \left \{ \mathcal{R}_n(\bar x_n+x -\bar x_n) \right \} \nonumber \\
& \leq - \log \left \{ \mathcal{R}_n(x) \right \},
\label{final_equality}
\end{align}
assuming only that $n$ is sufficiently large for
$n^{-1/2} \alpha t_0^\ast \leq 1$ to hold.
Hence $\mathcal{R}_n(x) \rightarrow 0$ in distribution as $n \rightarrow \infty$, due to  (\ref{first_of_three}) and (\ref{second_of_three}), and because $\epsilon$ and $A$ are arbitrary.
\end{proof}

\subsection{Proof of Proposition \ref{prop:ELbook_spine}}	
The proof uses the following result.
\begin{lemma}
\label{lem:concave}
The function $\x \mapsto \log(\mathcal R_n(\x))$ is concave on $M$.
\end{lemma}
\begin{proof}
When $\x=(x,k)$ lies in the interior $L^0_k \cong \mathbb R^p$ of page $k$, note that the folding map $F_k$ is linear, and the constraints in \eqref{eqn1d} thus lead to a convex feasibility set in the optimisation problem. Then, $\log \mathcal R_n|_{L^0_k}$ coincides with the EL function for $p$-dimensional Euclidean random variables $F_k(\x_i)$ which is concave \citep[e.g.][Theorem 3.2]{CharlesThesis}. This is true for any page $k$, and $\log \mathcal R_n$ is thus convex when restricted to $M\backslash S$. When $\x$ lies in $S$, a similar argument can be used for the $(p-1)$-dimensional Euclidean random variables $P_s(\x_i)$ to complete the proof. 
\end{proof}
\begin{remark}
It may be verified that the arguments in the proof above imply that $\log(\mathcal R_n) \circ \lambda:[0,1] \to \mathbb R$ is concave for every geodesic $\lambda:[0,1] \to M$, thereby verifying concavity of $\log \mathcal R_n$ as a function on the metric space $(M,d)$. 
\end{remark}
\begin{proof}
Choose a point $\x=(0, x^{(1)})$ on the spine $S$.  If $x^{(1)}$ does not lie in the convex hull of $P_s(\x_1)), \ldots, P_s(\x_n)$, then the feasible set for optimization problem \eqref{eqn1e} is empty, and $\log \left ( \mathcal{R}_n(\x\right )= - \infty$.  It is assumed that for the remainder of the proof that $x^{(1)}$ lies in the convex hull of $P_s(\x_1), \ldots , P_s(\x_n)$. 
	We first prove part (i). Note that (\ref{unconstrained}) has the same objective function as (\ref{eqn1e}) but the feasible region of (\ref{eqn1e}) is contained within the feasible region of (\ref{unconstrained}).  Hence if the solution to (\ref{unconstrained}), which is unique under the assumption that $x^{(1)}$ lies within the convex hull of $P_s(\x_1), \ldots , P_s(\x_n)$,  is feasible for (\ref{eqn1e}) then it is also the optimum for (\ref{eqn1e}).  Hence part (i) follows.
	
        Now consider part (ii). The feasible region, $\Omega$, for the optimization problem (\ref{eqn1e}) may be partitioned as follows: $\Omega=\bigcup_{k=0}^\ell \Omega_k$,
	where $\Omega_0$ consists of those $\{p_1, \ldots , p_n\}$ resulting in the constraints $\sum_i  p_i \langle F_k(\x_i), e_k \rangle \leqslant 0$, $k=1, \ldots , \ell$, being strict inequality constraints; and for $\Omega_j$, $j=1, \ldots , \ell$, constraint $\sum_{i=1}^n p_i \langle F_j(\x_i), e_j\rangle=0$ is an equality constraint and all the other inequality constraints are strict.  Since at most one of the inequalities in (\ref{eqn1e}) can be violated \citep{HHLMMMNOPS}, the $\left (\Omega_j \right )_{j=0}^\ell$ form a partition of $\Omega$. If the solution to (\ref{eqn1e}) lies in $\Omega_0$, then it must be a stationary maximum, which also must be a stationary maximum for optimization problem (\ref{unconstrained}), and by Lemma \ref{lem:concave} it must be the global maximum of (\ref{unconstrained}).  This corresponds to case (i).  If (\ref{test_case}) fails,  then by concavity of the objective function and convexity of the set $\Omega$, the only possibility is that the maximum for problem (\ref{eqn1e}) lies in the relative boundary of the convex set $\Omega$, i.e. one of the sets $\Omega_1, \ldots , \Omega_\ell$. Hence, in this case, the solution is given by (\ref{max_case}) and  the proof of part (ii) is now complete. 
\end{proof}	

\subsection{Proof of Theorem \ref{thm3}}
\begin{proof}
We first consider part (i).  Suppose that the population Fr\'echet mean is $\x_0=(x_0,k)$ where $x_0=(x_0^{(0)}, x_0^{(1)})$.  Then the relevant EL optimization problem is (\ref{eqn1d}) with $x=x_0$.  The proof follows from Theorem 2.1 for EL for Euclidean random variables $F_k(\x_1), \ldots , F_k(\x_n)$.

To establish part (ii), the key step is establishing that (\ref{test_case}) holds with probability approaching $1$ as $n \rightarrow \infty$, where in optimization problem  (\ref{unconstrained}) to obtain the $p_i^{[0]}$, we have set $x^{(1)}=x_0^{(1)}$, say, where 
\[
x_0^{(1)}= \int_{\M} P_s(\x) \d\mu (\x),
\]
is the population Fr\'echet mean of $P_s(\x)$.  Writing $\psi_0=x_0^{(1)}$ for convenience, we know that
 \[
p_i^{[0]} =\frac{1}{n} \frac{1}{\left (1+\langle \gamma, (P_s(\x_i) - \psi_0) \rangle \right )}
 \]
 where, provided $\psi_0$ lies in the convex hull of $P_s(\x_1, \ldots , P_s(\x_n)$, $\gamma \in \mathbb{R}^{p-1}$ is uniquely determined by
 \[
\sum_{i=1}^n \frac{1}{\left (1+\langle \gamma, (P_s(\x_i) - \psi_0) \rangle \right )}(P_s(\x_i) - \psi_0)=0_{p-1},
 \]
 where $0_{p-1}$ is the $(p-1)$-vector of zeros.  Expanding, assuming $\gamma$ is small, we see that the leading terms in the expansion of $\gamma$ is
 \begin{equation}
\gamma=n^{-1/2}\left [\bar{V}_0^{-1}\left \{n^{1/2} (\bar{\psi}_0-\psi_0) \right \} +r_n \right ],
\label{gamma_expansion}
 \end{equation}
 where the remainder term $r_n$ is of size $\vert \vert r_n \vert \vert = o_p(1)$,
and $\bar{\psi}_0$ is the sample mean of the $P_s(\x_i)$. 
Equation (\ref{gamma_expansion}) can be made rigorous without difficulty; see \cite[p.20]{ABO}.  Under the weak moment assumptions of the theorem, we are not able to make a stronger statement than $\vert \vert r_n \vert \vert = o_p(1)$ without further information.  Using (\ref{gamma_expansion}), the corresponding expansion for $p_i^{[0]}$ is 
\begin{equation}
p_i^{[0]}= n^{-1} \left \{ 1- n^{-1/2}(P_s(\x_i)-\psi_0)^\top \bar{V}_0^{-1} \left \{n^{1/2} (\bar{\psi}_0 - \psi_0) \right \} +o_p(n^{-1/2}) \ \right \},
\label{p_i_expansion}
\end{equation}
where
\[
\bar{V}_0 = n^{-1} \sum_{i=1}^n \left (P_s(\x_i) - \psi_0  \right ) \left ( P_s(\x_i) - \psi_0
\right )^\top.
\]
Consequently, for each $j=1, \ldots , \ell$,
\begin{align}
\sum_{i=1}^n & p_i^{[0]} \langle F_j(\x_i) , e_j \rangle 
=n^{-1}  \sum_{i=1}^n \langle F_j(\x_i) , e_j \rangle \nonumber \\
&-  n^{-1/2}\left \{ n^{-1}\sum_{i=1}^n \langle F_j(\x_i) , e_j \rangle (P_s(\x_i) - \psi_0) \right \}^\top  \bar{V}_0^{-1} \left \{ n^{1/2}(\bar{\psi}_0-\psi_0) \right \} +
o_p(n^{-1/2}).
\label{27B}
\end{align}

Since we are considering the sticky case, there exists an $\epsilon>0$ such that, for all $j=1, \ldots , \ell$,
\begin{equation}
a_j=\int_{\M} \langle F_j(\x), e_j \rangle \d\mu(\x) < -\epsilon <0.
 \label{a_j}
 \end{equation}
it follows from the strong law of large numbers applied to $n^{-1} \sum_{i=1}^n \langle F_j(\x_i), e_j\rangle$ plus the fact that the remainder terms on the RHS of (\ref{27B}) are $O_p(n^{-1/2})$, that
\begin{equation}
\pr \left [ \max_{j=1, \ldots , \ell} n^{-1} \sum_{i=1}^n p_i^{[0]} \langle F_j(\x_i), e_j \rangle  \leqslant - \epsilon/2 \right ] \rightarrow 1,
\label{less_than_epsilon_1}
\end{equation}
as $n \rightarrow \infty$.  Hence the condition (\ref{test_case}) is satisfied for each $j=1, \ldots, \ell$ and therefore the EL statistic $\log (\mathcal{R}_n(\x_0))$ from (\ref{eqn1e}) converges in probability to the statistic $\log(\mathcal{R}_n^{[0]}(\x_0))$ in (\ref{unconstrained}).  Moreover, we know that the latter correspond to the EL formulation for the mean of $(p-1)$-dimensional Euclidean random variables $P_s(\x_1), \ldots , P_s(\x_n)$ and so Theorem 2.1 can be applied.  Hence part (ii) is proved.

Finally, we consider part (iii), the semi-sticky case.  Suppose that $k$ is the unique $j=1, \ldots , \ell$ such that (\ref{a_j}) is zero.  Define $x_0=(0,x_0^{(1)})$.  Note that $x_0$ is the population Euclidean mean in the case of $\mathcal{R}_n^{[k]}(\cdot)$ in (\ref{constrained_plus_one}) but $x_0$ is not the relevant population Euclidean mean in the case of $\mathcal{R}_n^{[j]}(\cdot)$ in (\ref{constrained_plus_one}), $j \neq k$.  Specifically, the Euclidean mean in these latter cases is $(a_j, x_0^{(1)})$ where $a_j <0$ for $j \neq k$; see (\ref{a_j}).  Hence, because $-\log  ( \mathcal{R}_n^{[k]}(x_0)   )$ has an asymptotic $\chi_p^2$ distribution and, by Theorem \ref{thm:alternative}, 
\[
\log \left ( \mathcal{R}_n^{[j]}(x_0)  \right ) \rightarrow - \infty
\]
in probability for each $j \neq k$, we conclude that, with probability approaching $1$, 
\[
\log \left ( \mathcal{R}_n^{[k]}( x_0) \right ) > \max_{j \neq k} \log \left ( \mathcal{R}_n^{[j]}(x_0) \right ).
\]
Consequently, it follows from Proposition \ref{prop:ELbook_spine} that the maximum of the objective function occurs either on page $k$ or on the spine, and in the former case we choose $j=k$ in (\ref{max_case}) with probability approaching 1.  Moreover, it follows from symmetry properties of the normal distribution that (i) it is on page $k$ with probability approaching $1/2$ and on the spine with probability approaching $1/2$; and (ii) the limiting conditional distribution in the former case is $\chi_p^2$ and in the latter case, $\chi_{p-1}^2$.  So part (iii) is now proved.    
\end{proof}

\subsection{Proof of Theorem \ref{thm:bootstrap}}
\begin{proof}
A  sketch of a method of proof, applicable in a wide range of settings including that of Theorem \ref{thm:bootstrap},  is outlined in Appendix B of \cite{FHJW96}.  However, although the proof sketch in their paper provides the key steps and insights in the approach, it does not give precise enough information to identify explicit sufficient conditions for Theorem \ref{thm:bootstrap} to hold.  Our goal here is not to give a full proof of Theorem \ref{thm:bootstrap}, which largely follows  from results in the literature due to Peter Hall and coauthors, but rather to derive the explicit sufficient conditions that are stated in Theorem \ref{thm:bootstrap}.

In a lengthy calculation, \cite{DHR91} prove thet EL is Bartlett correctable.  Although that is a different result to Theorem \ref{Theorem_5.1} (specifically, the content of Theorem \ref{Theorem_5.1} does not depend on EL being Bartlett correctable), nevertheless the  expansions derived in \cite{DHR91} provide clarification regarding what conditions are sufficient  for Theorem \ref{Theorem_5.1} to hold.  

First, we consider the expansion (4.1) in \cite{DHR91}, which is different to the one relevant to Theorem \ref{Theorem_5.1} in three respects: first, we need to consider all terms up to and including those of  nominal size $O_p(n^{-3})$;  second, we only need to consider Euclidean vector means, as opposed to the `smooth function model' considered in \cite{DHR91}; and third, we need to view the components $X^j$ as being functions of $\x$, where $\x$ lies within a neighbourhood of $\x_0 \in \M$. The nominal size of each term in the expansion  is  $n^{-r/2}$ where $r$ is the number of $A$-terms multiplied together.  Moreover, the nominal size is equal to the actual size, e.g. 
$A^{jk}(\x_0)A^j (x_0)A^k (\x_0) = O_p(n^{-3/2})$, if and only if $\textrm{Var}(A^{jk}(\x_0))=O(n^{-1})$ or, equivalently, if $\textrm{Var}(X^j (\x_0) X^k (\x_0))$ is finite; and e.g. the nominal size, $O_p(n^{-2})$,  of $A^{jk\ell} (\x_0)A^j (x_0)A^k (\x_0) A^\ell(\x_0)$  is the actual size if and only if $\textrm{Var}(A^{jk\ell}(\x_0))=O(n^{-1})$ or, equivalently, if $\textrm{Var}(X^j (\x_0) X^k (\x_0) X^\ell (\x_0))$ is finite.    Inspection of the terms up to and including nominal size $O_p(n^{-3})$ shows that the terms involving the highest moments, in fact fifth moments, are of the form
\[
A^{jk\ell m n}(\x_0)A^j (\x_0) A^k (\x_0)A^{\ell}(\x_0) A^m (\x_0) A^n (\x_0).
\]
For $A^{jk\ell mn}(\x_0)=O_p(n^{-1/2})$ to hold, and therefore for the displayed term immediately above to be of size $O_p(n^{-3})$, it is necessary and sufficient that $\textrm{Var}[A^{jk\ell m n}(\x_0)]=O(n^{-1})$ for all choices of $j,k,\ell, m , n=1, \ldots , p$ which, because $p$, the dimension of $\M$, remains bounded, is equivalent to condition  (\ref{Condition_2}).  

Second, since we are employing boostrap calibration of EL, the $n^{-1}$ term in the expansion is proportional to $n^{-1}\tau (\bar x_n)$ rather than $n^{-1} \tau (\x_0)$, where as before $\bar \x_n$ is the sample Fr\'echet mean.  Under the assumption that $\tau$ has a continuous Hessian in a neighbourhood of $X_0$, we have a Taylor expansion of the form
\begin{equation}
\tau(\bar \x_n)=\tau (\x_0) +n^{-1/2}F_1(\x_0) + n^{-1}F_2(\x_0^\ast),
\label{s27}
\end{equation}
where $\x_0^\ast$ lies on the geodesic connecting $\bar \x_n$ and $\x_0$, where the final term on the RHS of (\ref{s27}) is of the stated order, $O_p(n^{-1})$, due to  the $\sqrt{n}$-consistency of $\bar \x_n$ as an estimator of $\x_0$ and the continuity of the Hessian of $\tau$ at $\x_0$.  Moreover, in the asymptotic expansion of the coverage probability of the bootstrap-calibrated EL confidence region, the contribution of the final term on the RHS of (\ref{s27}) is of size $O(n^{-2})$, due to $\tau$ being multiplied by $n^{-1}$.  

Third, and finally, we address the  question of the size of the contribution of the middle term on the RHS of (\ref{s27}) to the asymptotic expansion of the coverage probability.  
After muliplying by $n^{-1}$, it appears that this term should be of size $n^{-3/2}$.  However, after close scrutiny, it turns out that this term is $O(n^{-2})$, due to the parity properties of the polynomials (Hermite polynomials)  that arise in Edgeworth expansions.  A discussion of this phenomenon in the case of Bartlett correction of parametric likelihood for continuous models is given by \cite{BNH}; the same argument works in other settings provided the relevant population cumulants have expansions in powers of $n^{-1}$, as is the case in the present context.
\end{proof}
\end{appendix}

%
%
%


\bibliographystyle{abbrvnat}
\bibliography{biblio}

\end{document}